\documentclass{amsart}

\usepackage{amssymb,amsmath,amsthm,amsfonts,commath,microtype}
\usepackage[T1]{fontenc}
\usepackage[utf8]{inputenc}
\everymath{\displaystyle}
\usepackage{mathrsfs}
\usepackage{indentfirst}
\usepackage[titletoc,title]{appendix}

\usepackage{paralist}
\usepackage{ascmac}
\usepackage[dvipsnames]{xcolor}
\usepackage{graphicx}
\usepackage{enumitem,comment,wrapfig}
\usepackage[subrefformat=parens]{subcaption}
\usepackage{autobreak}

\graphicspath{{./Pictures/}}

\usepackage{tikz}
\usetikzlibrary{
	calc,
	arrows,
	patterns,
	intersections,
	decorations.pathreplacing,
	angles
	}

\theoremstyle{definition}
\newtheorem{definition}{Definition}[section]
\newtheorem{theorem}[definition]{Theorem}
\newtheorem{lemma}[definition]{Lemma}
\newtheorem{corollary}[definition]{Corollary}

\theoremstyle{remark}
\newtheorem{remark}[definition]{Remark}

\makeatletter

\@addtoreset{equation}{section}
\makeatother

\usepackage{hyperref}
\hypersetup{
	setpagesize = false,
	colorlinks = true,
	citecolor = red,
	linkcolor = red,
	urlcolor = red,
	pdftitle = {},
	pdfsubject = {},
	pdfauthor = {},
	pdfkeywords = {}
	}

\usepackage[
	initials,
	msc-links,
	nobysame
	]{amsrefs}[2007/10/22]
\BibSpec{article}{%
	+{}  {\PrintAuthors}                {author}
	+{,} { \TitleWithUrl}               {title}
	+{.} { }                            {part}
	+{:} { \textit}                     {subtitle}
	+{,} { \PrintContributions}         {contribution}
	+{.} { \PrintPartials}              {partial}
	+{,} { }                            {journal}
	+{}  { \textbf}                     {volume}
	+{}  { \PrintDate}                  {date}
	+{,} { \issuetext}                  {number}
	+{,} { \eprintpages}                {pages}
	+{,} { }                            {status}
	+{,} { available at arXiv:\eprint}  {eprint}
	+{}  { \parenthesize}               {language}
	+{}  { \PrintTranslation}           {translation}
	+{;} { \PrintReprint}               {reprint}
	+{.} { }                            {note}
	+{.} {}                             {transition}
	+{}  {\SentenceSpace \PrintReviews} {review}
}

\BibSpec{collection.article}{%
	+{}  {\PrintAuthors}                {author}
	+{,} { \TitleWithUrl}               {title}
	+{.} { }                            {part}
	+{:} { \textit}                     {subtitle}
	+{,} { \PrintContributions}         {contribution}
	+{,} { \PrintConference}            {conference}
	+{}  {\PrintBook}                   {book}
	+{,} { }                            {booktitle}
	+{,} { \PrintDateB}                 {date}
	+{,} { pp.~}                        {pages}
	+{,} { }                            {status}
	+{,} { available at \eprint}        {eprint}
	+{}  { \parenthesize}               {language}
	+{}  { \PrintTranslation}           {translation}
	+{;} { \PrintReprint}               {reprint}
	+{.} { }                            {note}
	+{.} {}                             {transition}
	+{}  {\SentenceSpace \PrintReviews} {review}
}
\BibSpec{book}{%
	+{}  {\PrintPrimary}                {transition}
	+{,} { \TitleWithUrl}               {title}
	+{.} { }                            {part}
	+{:} { \textit}                     {subtitle}
	+{,} { \PrintEdition}               {edition}
	+{}  { \PrintEditorsB}              {editor}
	+{,} { \PrintTranslatorsC}          {translator}
	+{,} { \PrintContributions}         {contribution}
	+{,} { }                            {series}
	+{,} { \voltext}                    {volume}
	+{,} { }                            {publisher}
	+{,} { }                            {organization}
	+{,} { }                            {address}
	+{,} { \PrintDateB}                 {date}
	+{,} { }                            {status}
	+{}  { \parenthesize}               {language}
	+{}  { \PrintTranslation}           {translation}
	+{;} { \PrintReprint}               {reprint}
	+{.} { }                            {note}
	+{.} {}                             {transition}
	+{}  {\SentenceSpace \PrintReviews} {review}
}

\BibSpec{thesis}{%
	+{}  {\PrintAuthors}                {author}
	+{.} { \TitleWithUrl}               {title}
	+{:} { \textit}                     {subtitle}
	+{,} { \PrintThesisType}            {type}
	+{,} { }                            {organization}
	+{} { \PrintDate}                   {date}
	+{,} { }                            {address}
	+{,} { \eprint}                     {eprint}
	+{,} { }                            {status}
	+{}  { \parenthesize}               {language}
	+{}  { \PrintTranslation}           {translation}
	+{;} { \PrintReprint}               {reprint}
	+{.} { }                            {note}
	+{.} {}                             {transition}
	+{}  {\SentenceSpace \PrintReviews} {review}
}

\newcommand{\TitleWithUrl}[1]{
	\IfEmptyBibField{doi}%
	{\IfEmptyBibField{url}{\textit{#1}}%
	{\IfEmptyBibField{eprint}{\href {\BibField{url}}{\textit{#1}}}{\textit{#1}}}%
	}%
	{\href {https://doi.org/\BibField{doi}}{\textit{#1}}}
	}
\renewcommand{\eprint}[1]{
	\IfEmptyBibField{url}{\url{#1}}%
	{\href {\BibField{url}}{#1}}
	}

\title[Lorentz Darboux transformations of curves]{Darboux transformations of spacelike curves in the Lorentz-Minkowski plane}

\author{Masaya Hara}
\address[Masaya Hara]{Department of Mathematics, Graduate School of Science, Kobe university, 1-1 Rokkodai-cho, Nada-ku, Kobe 657-8501, Japan}
\email{mhara@math.kobe-u.ac.jp}


\begin{document}

\begin{abstract}
	This paper concerns (Lorentz-)Darboux transformations in the Lorentz-Minkowski plane $\mathbb{R}^{1,1}$. We use the Penrose diagram for conformal compactification and show some unique properties of Darboux transformations of spacelike curves in $\mathbb{R}^{1,1}$, especially with regard to singularities and blowup. Finally, we make some comments on Darboux transformation of type-changing curves.
\end{abstract}

\maketitle

\section*{Introduction}
Darboux transformations have been considered for studying some special classes of surfaces such as isothermic surfaces and constant mean curvature surfaces etc, since they preserve conformal curvature line coordinates.
Recently, Darboux transformations are also used for studying curves in Euclidean spaces, see for example \cites{MR4287306}.
In this paper, we focus on Darboux transformations of planar curves in the Lorentz-Minkowski plane.
In particular, we aim to investigate the singularities and blowups of Darboux transformations of spacelike curves and understand them geometrically.

According to \cites{MR175590, MR149912}, Lorentz-Minkowski spaces can be conformally compactified (called a \emph{Penrose diagram}), and we have types of infinities, see Definition~\ref{Definition:PenroseDiagram} (see also \cites{MR3792084, MR172678, MR1426141}).
We introduce \emph{split complex numbers}, which are identified with $\mathbb{R}^{1,1}$, and define (Lorentz-)Darboux transformations, see Definition~\ref{Definition:SpacelikeDarbouxCrossRatio}.
By using types of infinities from the Penrose diagram, we prove our main result Theorem~\ref{Theorem:NullInfinityOrthogonally}, which asserts that we can control the directions of Darboux transformations at infinities by using the arc-length polarization.
Finally, we give a brief discussion on Darboux transformations of type-changing curves.
\section{Preliminaries}
In this section, we review the geometry of the Lorentz-Minkowski plane. The basic analysis of the Lorentz-Minkowski plane is explained in \cites{MR2141751} in detail, and here we only include what we use in this paper.

\subsection{Lorentz-Minkowski plane}
Let $\mathbb{R}^{1,1}$ denote a 2-dimensional pseudo-Riemannian space with inner product $\langle \cdot , \cdot \rangle$ of signature $\left(+ -\right)$, that is, for $x, y \in \mathbb{R}^{1,1}$,
\[\langle x,y \rangle=\langle \left(x_1,x_2\right)^t,\left(y_1,y_2\right)^t \rangle=x_1y_1-x_2y_2,\]
and we use $|\cdot|^2$ to denote the inner product of a vector with itself, i.e. the "length squared" in $\mathbb{R}^{1,1}$, which in this case can be nonpositive.

\subsection{Split complex plane}
The geomety of the Lorentz-Minkowski plane can be modelled by the \emph{split complex plane} $\mathbb{C}'$ defined as
\[\mathbb{C}' :=\{a+j b \mid a,b\in \mathbb{R},\ j^2=1,\ j \notin \mathbb{R}\}.\]

We identify $\mathbb{R}^{1,1}$ with $\mathbb{C}'$ by
\begin{equation*}
	\mathbb{C}' \ni a+jb \mapsto \left(a,b\right) \in \mathbb{R}^{1,1}.
\end{equation*}
Note that $\mathbb{C}'$ is not a field but a commutative ring, that is, the inverse of $x=x_1+j x_2 \in \mathbb{C}'$ given by
\[x^{-1}=\dfrac{\bar{x} }{|x|^2}=\dfrac{x_1-jx_2}{x_1^2-x_2^2},\]
does not exist when $x_1=\pm x_2$. The algebra $\mathbb{C}'$ admits such \emph{zero divisors} \cites{MR3226609}. We denote by $A$ the set consisting of the zero divisors and the zero point (0) of $\mathbb{C}'$, that is,
\[A=\{a\pm j a \mid a \in \mathbb{R},\ j^2=1,\ j \notin \mathbb{R}\} \subset \mathbb{C}'.\]
By applying the above identification, we consider split complex numbers to have signature type as well, and so the inverse of a lightlike split complex number is undefined.

\begin{remark}
	In this paper, we use the naming "split complex numbers", which is taken from \cites{JosephMthesis, MR0056937} for example. It is also named Lorentz numbers, para-complex numbers, hyperbolic numbers etc. For further information, see \cites{MR3226609, MR2141751}.
\end{remark}

\subsection{Penrose diagram}

\begin{wrapfigure}[12]{r}[0mm]{4.2truecm}
	\vspace{-1\baselineskip}
	\begin{tikzpicture}[scale=1.5]
		\draw(0,1)--(-1,0)--(0,-1)--(1,0)--(0,1);
		\draw(0,1)node[above=3pt]{$I^+$};
		\draw(-1,0)node[left=3pt]{$I^0$};
		\draw(0,-1)node[below=3pt]{$I^-$};
		\draw(1,0)node[right=3pt]{$I^0$};
		\draw(0.8,0.8)node{$\mathscr{I}^+$};
		\draw(-0.75,0.75)node{$\mathscr{I}^+$};
		\draw(-0.75,-0.75)node{$\mathscr{I}^-$};
		\draw(0.75,-0.75)node{$\mathscr{I}^-$};
		\draw[decorate,decoration={brace, mirror, amplitude=2mm, raise=3pt}](0,1)--(-1,0);
		\draw[decorate,decoration={brace, mirror, amplitude=2mm, raise=3pt}](-1,0)--(0,-1);
		\draw[decorate,decoration={brace, mirror, amplitude=2mm, raise=3pt}](0,-1)--(1,0);
		\draw[decorate,decoration={brace, mirror, amplitude=2mm, raise=3pt}](1,0)--(0,1);
		\draw[thick, ->, >=stealth](-1,0)--(0.95,0);
		\draw[thick, ->, >=stealth](0,-1)--(0,0.95);
		\draw(0.8,-0.05)node[above left]{$\psi$};
		\draw(0.05,0.85)node[below left]{$\zeta$};
		\draw(0,0)node[below left]{$O$};
	\end{tikzpicture}
	\caption{Penrose diagram.}\label{Picture:PenroseDiagram}
\end{wrapfigure}
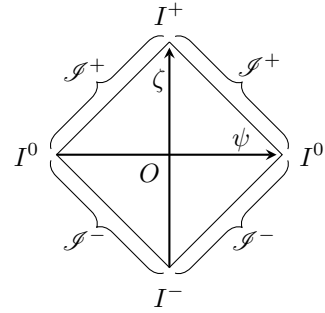

Next, we introduce a conformal compactification of $\mathbb{R}^{1,1} \cong \mathbb{C}'$ called the \emph{Penrose diagram}. It is named after the mathematical physicist Roger Penrose and is well known in theoretical physics. (See \cites{MR3792084,MR1426141}.)

First of all, we make a change of coordinates for $(x,y)$ in $\mathbb{R}^{1,1}$ to $u=x+y$ and $v=x-y$, which is called \emph{null coordinates} (or \emph{light-cone coordinates}). Then $u$ and $v$ take values in $-\infty < u,v < \infty$.
Now, we define $\psi$ and $\zeta$ by
\begin{equation*}
	\displaywidth=\parshapelength\numexpr\prevgraf+2\relax
	u=\tan \frac{\psi+\zeta}{2}\ \text{and}\ v=\tan \frac{\psi-\zeta}{2}.
\end{equation*}
Then $-\pi \leq \psi \pm \zeta \leq \pi$ and we have
\[dx^2-dy^2=dudv=\Omega^2(d\psi^2-d\zeta^2),\]
where $\Omega^{-2}=4\cos^2 \frac{\psi+\zeta}{2} \cos^2 \frac{\psi-\zeta}{2}$.

Since $\Omega^2$ vanishes only on the boundary of the square made by the values of $(\psi, \zeta)$, and is positive in the interior, we obtain a conformal compactification into the interior of the square.

We summarize the above fact as a definition and define types of infinity as used in relativity.

\begin{definition}[\cites{MR3792084, MR175590, MR149912, MR172678, MR1426141}]\label{Definition:PenroseDiagram}
	Let $\mathcal{P}$ be a map from $\mathbb{R}^{1,1}$ to $\mathbb{R}^{1,1}$ defined by $\mathcal{P}(x,y)=(\psi,\zeta)$, where
	\begin{equation*}
		\left\{
			\begin{alignedat}{2}
				\psi & = \arctan(x+y)&  + & \arctan(x-y) \\[1pt]
				\zeta  & = \arctan(x+y) & - & \arctan(x-y).
			\end{alignedat}
		\right.
	\end{equation*}

	Then $\mathcal{P}$ is a conformal mapping of $\mathbb{R}^{1,1}$ into the interior of a square.
	Diagrams made by this mapping are called \emph{Penrose diagrams}.
	The boundary of a Penrose diagram consists of points at infinity, and there are types of infinities as follows:
	In the Figure~\ref{Picture:PenroseDiagram} to the above,
	\begin{itemize}
		\item the point $I^+ (I^-)$ is called the \emph{future (past) timelike infinity},
		\item the point $I^0$ is called the \emph{spatial infinity} and
		\item the edges $\mathscr{I}^+  (\mathscr{I}^-)$ are called the \emph{future (past) null infinity}.
	\end{itemize}
\end{definition}

\begin{remark}
	In this paper, we do not need to distinguish the future from the past, so we just use three types of infinities; \emph{spatial infinity} ($I_s$), \emph{timelike infinity} ($I_t$) and \emph{null infinity} ($\mathscr{I}$).
	Also note that a line segment of null infinity is open, that is, it does not include the spatial and timelike infinities at its ends.
\end{remark}
\section{Lorentz-Darboux transformations}
In this section, we would like to introduce Darboux transformations in the Lorentz-Minkowski plane, especially with regard to singularities and blow up.
From now on, we only consider smooth curves $x(t)$ on $\mathbb{C}'$.

\subsection{Spacelike curves}
We suppose that a curve $x(t):I \rightarrow \mathbb{C}'$ is spacelike, that is, $|x'(t)|^2>0$ for any $t \in I$, where $I \subset \mathbb{R}$ is an interval parametrized by $t$.
Further, we consider \emph{polarized curves} by including a non-vanishing quadratic differential $\tfrac{dt^2}{m}$ with $m(t):I \rightarrow \mathbb{R}^{\times}:=\mathbb{R}\backslash\{0\}$ (see \cites{MR3936232}).

The polarization $m(t)$ is called an \emph{arc-length polarization} for $x(t)$ if
\[m(t)\equiv \frac{1}{|x'(t)|^2}.\]
We introduce an extension of Darboux transformations, as follows.

On the complex plane, it is well known that Darboux transformations can be defined as satisfying an \emph{infinitesimal cross ratio} condition
\begin{equation}\label{Equation:TrueCrossRatio}
	\frac{x'\hat{x}'}{(\hat{x}-x)^2} = \frac{\lambda}{m}.
\end{equation}
In the split complex plane, however, $\hat{x}-x$ might be a zero divisor, so we need to avoid putting $\hat{x}-x$ in the denominator. Therefore we adopt the next equation as the definition of Darboux transformations in the split complex plane.
\begin{definition}\label{Definition:SpacelikeDarbouxCrossRatio}
	Let $x(t):\left(I, \tfrac{dt^2}{m}\right) \rightarrow \mathbb{C}'$ be a smooth spacelike curve polarized by $m(t): I \rightarrow \mathbb{R}^{\times}$. For some nonzero constant $\lambda \in \mathbb{R}^{\times}$, a curve $\hat{x}(t)$ satisfying
	\begin{equation}\label{Equation:CrossRatio}
		x'\hat{x}' = \frac{\lambda}{m}(\hat{x}-x)^2
	\end{equation}
	is called a \emph{(Lorentz-)Darboux transformation with parameter $\lambda$} of $x(t)$.
\end{definition}

\begin{figure}[hbt]
	\centering
	\begin{minipage}[b]{0.35\columnwidth}
		\centering
		\includegraphics[keepaspectratio, width=\columnwidth]{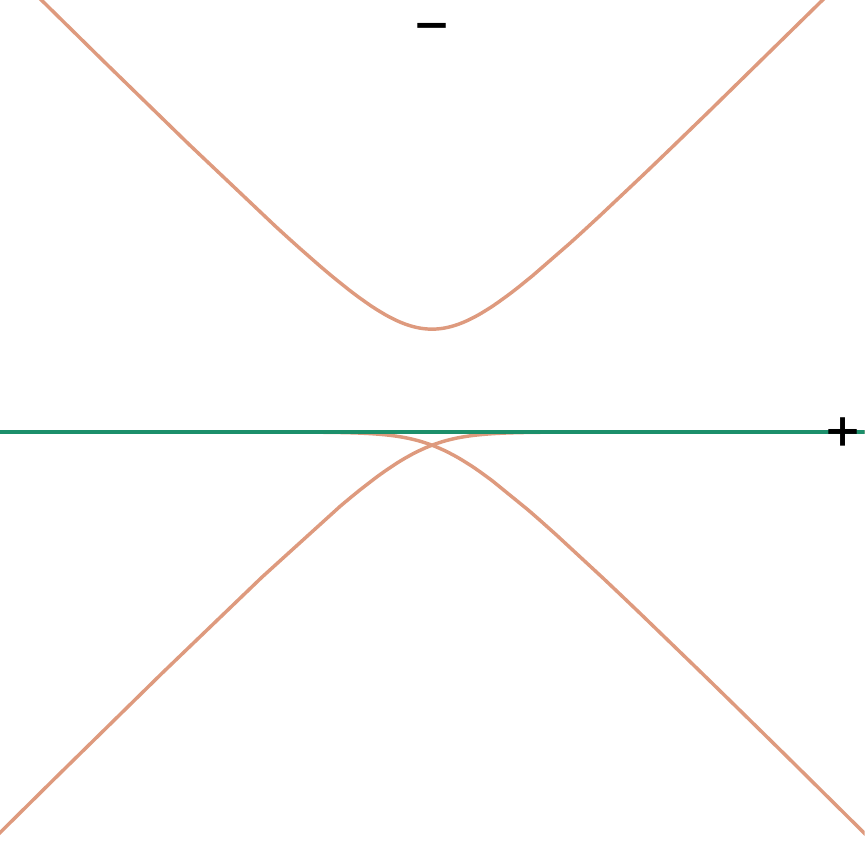}
		\subcaption{\upshape $\lambda >0$ case.}\label{DT1}
	\end{minipage}
	\hspace{0.1\columnwidth}
	\begin{minipage}[b]{0.35\columnwidth}
		\centering
		\includegraphics[keepaspectratio, width=\columnwidth]{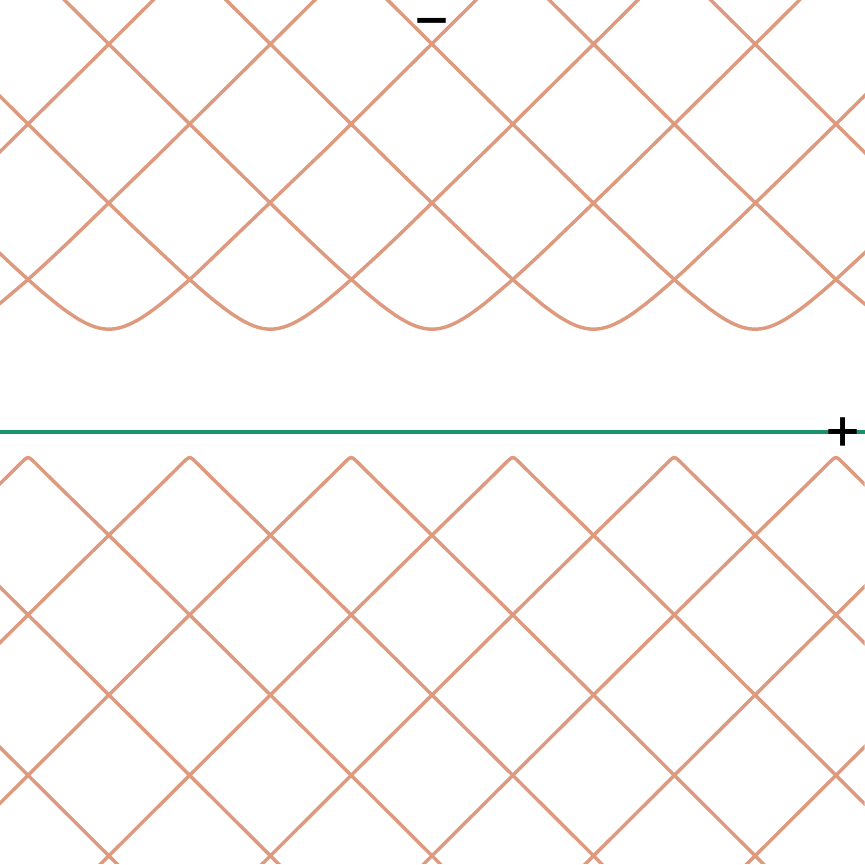}
		\subcaption{\upshape $\lambda <0$ case.}\label{DT2}
	\end{minipage}
	\caption{Darboux transformations of the horizontal line $x(t)=(t,0)$ in $\mathbb{C}'$.}\label{Picture:DarbouxTransformations}
\end{figure}

\begin{figure}[hbt]
	\centering
	\begin{minipage}[b]{0.375\columnwidth}
		\centering
		\includegraphics[keepaspectratio, width=\columnwidth]{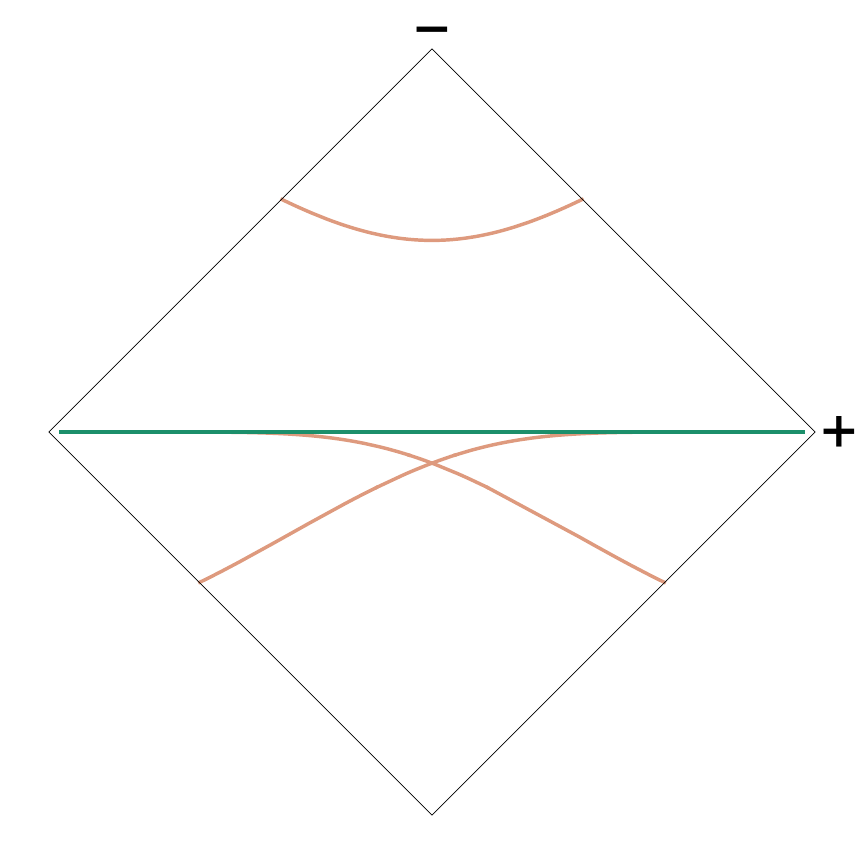}
		\subcaption{\upshape $\lambda >0$ case.}\label{PenroseDT1}
	\end{minipage}
	\hspace{0.05\columnwidth}
	\begin{minipage}[b]{0.375\columnwidth}
		\centering
		\includegraphics[keepaspectratio, width=\columnwidth]{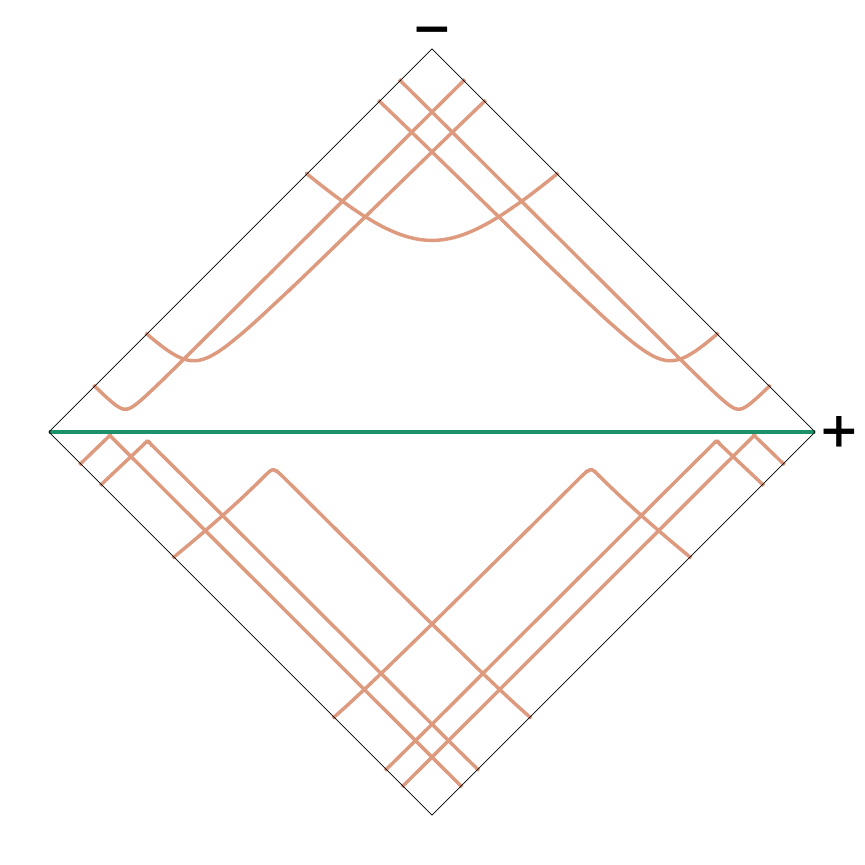}
		\subcaption{\upshape $\lambda <0$ case.}\label{PenroseDT2}
	\end{minipage}
	\caption{Darboux transformations of the horizontal line in $\mathbb{C}'$ parametrized by $x(t)=(t,0)$, now presented in the Penrose diagram.}\label{Picture:DarbouxTransformationsInPenroseDiagram}
\end{figure}

One of the essential characterizations of a Darboux pair is that there exists a \emph{common circle congruence}, that is, both members of the pair touch a circle tangentially at corresponding points for each $t$.
We have the same properties for a Lorentz-Darboux pair as well, as in the following discussion.

\begin{remark}
	In the equation~\eqref{Equation:CrossRatio}, $\hat{x}-x$ can be lightlike when $x'$ is spacelike.
	Then, at such a point, the common circle congruence to $x$ and $\hat{x}$ becomes a lightlike circle. We call such a point a \emph{degenerate point}. One can consider the Lorentz-Darboux transformation by using the original equation~\eqref{Equation:TrueCrossRatio} and taking limits at degenerate points. Furthermore, the equation~\eqref{Equation:CrossRatio} suggests that $\hat{x}-x$ can be 0, and then $\hat{x}=x$. But we omit such cases.
\end{remark}

\begin{lemma}\label{Lemma:Radius}
	Let $x$ and $\hat{x}$ be a Darboux pair in $\mathbb{C}'$ and $n$ be a normal of $x=(x_1,x_2)$ given by $n=(x_2',x_1')$. Then the radius $\xi$ of the common circle congruence is given by
	\[\xi=\frac{|\hat{x}-x|^2}{2\langle \hat{x}-x,n \rangle}.\]
\end{lemma}
\begin{proof}
	Without loss of generality, we suppose that $|x'|^2=1$. Then we have the common circle congruence with radius $\xi$, which means
	\[x-\frac{n}{\sqrt{-|n|^2}}\xi=\hat{x}+\frac{\hat{n}}{\sqrt{-|\hat{n}|^2}}\xi,\]
	where $\hat{n}$ is a normal vector to $\hat{x}'$ and defined as like for $n$ of $x'$. Therefore,
	\[ \xi = \frac{|\hat{x}-x|^2}{\Big\langle \frac{n}{\sqrt{|x'|^2}}+\frac{\hat{n}}{\sqrt{|\hat{x}'|^2}}, \hat{x}-x \Big\rangle}.\]
	Here, if we write $\hat{x}'=P\left(|\hat{x}-x|^2x'-2 \langle \hat{x}-x,x' \rangle(\hat{x}-x) \right)$ so that $P=-\frac{\lambda}{m|x'|^2}$, we can verify that
	\begin{align*}
		\bigg\langle \frac{n}{\sqrt{|x'|^2}}+\frac{\hat{n}}{\sqrt{|\hat{x}'|^2}}, \hat{x}-x \bigg\rangle
		&= \langle n,\hat{x}-x \rangle +\frac{\langle\hat{n}, \hat{x}-x \rangle}{\sqrt{|\hat{x}'|^2}}\\
		&= \langle n,\hat{x}-x \rangle
		+\frac{P|\hat{x}-x|^2\langle n, \hat{x}-x \rangle}{\sqrt{P^2|\hat{x}-x|^4|x'|^2}}\\
		&= 2\langle n,\hat{x}-x \rangle.
	\end{align*}
	Thus we conclude that $\xi =\frac{|\hat{x}-x|^2}{2\langle \hat{x}-x,n \rangle}$.
\end{proof}

To close this subsection on Darboux transformations and polarization, we also give a lemma about arc-length polarization:

\begin{lemma}[\cites{MR4287306}]
	Let $x,\hat{x} : \left(I,\tfrac{dt^2}{m}\right) \rightarrow \mathbb{C}'$ be a Darboux pair with parameter $\lambda$,
	where $\tfrac{dt^2}{m}$ is the arc-length polarization of $x$. Then $\tfrac{ds^2}{m}$ is also the arc-length polarization of $\hat{x}$ if and only if $|\hat{x}-x|^2=\tfrac{1}{\lambda}$ at one point $t_* \in I$.
	Furthermore, if $|\hat{x}-x|^2=\tfrac{1}{\lambda}$ at one point, then $|\hat{x}-x|^2$ is constant for all points.
\end{lemma}
The above lemma can be proven simply by changing the metric in the proof in \cites{MR4287306}.

\subsubsection{Singularities and blow up}
For a Darboux transformation of a spacelike curve, we take a polarization. We say a curve $x(t):I \rightarrow \mathbb{R}^{1,1}$ has a singular point at $t=0$ if $x'(t)=0$ at $t=t_*$. Since the polarization $m$ does not blow up and $\hat{x}-x \neq 0$ in the domain, we have the same fact as for Darboux transformations in $\mathbb{R}^2$:

\begin{theorem}
	If the curve $x(t)$ is spacelike, then its Darboux transformation $\hat{x}(t)$ does not have singular points.
\end{theorem}

Next we would like to observe the behaviors of Darboux transformations at infinity.
\begin{definition}
	We say a curve $x(t):I \rightarrow \mathbb{R}^{1,1}$ blows up at $t=t_*$ if $\lim_{t \rightarrow t_*}\mathcal{P}(x(t))$ is at the boundary of the Penrose diagram, and we define the direction of $x'(t)$ at infinity as $\lim_{t \rightarrow t_*}\mathcal{P}(x(t))'$, which we can allow to be infinite as we are interested only in the direction.
\end{definition}
\begin{definition}
	Let two curves $x(t)$ and $y(t)$ be at infinity as $t  \rightarrow t_*$. They are said to be tangent at infinity if, after some inversion, $x$ and $y$ are tangent at a finite point.
\end{definition}

Before going to the theorems, we prove two preparatory results, which provide intuition on the Penrose diagram.

\begin{lemma}\label{Lemma:NullInfinity}
	Any nonlinear circle in the Penrose diagram blows up at the null infinity.
\end{lemma}
\begin{proof}
	Without loss of generality, in the timelike case we can just consider a timelike circle given by
	\[x(t)=(r \sec(t)+c_1, r \tan(t)+c_2)\]
	for some finite $c_1, c_2 \in \mathbb{R}$ and $r \in \mathbb{R}_{>0}$.
	(Note that the "nonlinear circle" in the Lorentz-Minkowski plane is just a special case of hyperbolas.)
	Then $x(t)$ represents a timelike circle with center $(c_1, c_2)$ and radius $r$.\\
	Now, in the Penrose diagram, we have
	\[\mathcal{P}(x)=(\arctan(\theta_+)+\arctan(\theta_-), \arctan(\theta_+)-\arctan(\theta_-)),\]
	where
	\[\theta_\pm=c_1 \pm c_2+r \sec(t)\pm r \tan(t).\]

	This curve $x(t)$ blows up as $t \rightarrow \pi/2$ for example, and the next equations hold:
	\[\lim_{t \rightarrow \pi/2}(\sec(t)+\tan(t))= \infty \quad \text{and} \quad \lim_{t \rightarrow \pi/2}(\sec(t)-\tan(t))= 0.\]
	This means that $\mathcal{P}(x(t))$ converges to an interior point of the upper right boundary in the Penrose diagram. We can consider $t \rightarrow \pi/2 + k\pi$ for $k \in \mathbb{Z}$ in the same manner. Therefore $\mathcal{P}(x(t))$ blows up at the null infinity.
	Analogous arguments hold for lightlike and spacelike circles, and analogous $c_1, c_2$ and $r$ appear in the spacelike case:
	\[x(t)=\left(r \tan(t)+c_1, r \sec(t)+c_2\right).\]
	In the lightlike case we have
	\[x(t)=(t+c_1, t+c_2).\]
\end{proof}

The following corollary to Lemma~\ref{Lemma:NullInfinity} is helpful for giving us a better sense of the Penrose diagram, although it is not directly related to the proofs of the upcoming theorems, but is morally connected to Theorem~\ref{Theorem:NullInfinityOrthogonally}.

\begin{corollary}\label{Cororlaryy:Orthogonally}
	Suppose that a smooth non-singular curve $x(t)$ blows up as $t \rightarrow t_*$. If $\lim_{t \rightarrow t_*}|x'(t)|^2=0$, then $x(t)$ blows up at the null infinity in a lightlike direction in the Penrose diagram as $t \rightarrow t_*$.
\end{corollary}
\begin{proof}
	Without loss of generality, we write a smooth curve $x(t)=(f(t), g(t))$ for some smooth functions $f(t)$ and $g(t)$.
	Since $\lim_{t \rightarrow t_*}|x'(t)|^2=0$, we have
	\[\lim_{t \rightarrow t_*}\left(f'(t)^2-g'(t)^2\right)=0.\]
	Considering the inversion $\sigma(x)=\frac{x}{|x|^2}$, a direct computation leads to
	\[|(\mathcal{P}(\sigma(x(t))))'|^2=4\frac{f'^2-g'^2}{\Theta},\]
	where
	\[\Theta=\left(1+(f+g)^2\right)\left(1+(f-g)^2\right).\]

	Since $\lim_{t \rightarrow t_*}\Theta\neq 0$, we have $\lim_{t \rightarrow t_*}|(\mathcal{P}(\sigma(x(t))))'|^2=0$. Furthermore, $\mathcal{P}$ and $\sigma$ are conformal maps, therefore we conclude that $x(t)$ blows up at the null infinity in a lightlike direction in the Penrose diagram as $t\rightarrow t_*$.
\end{proof}
We can now prove two theorems about Darboux transformations as follows:

\begin{theorem}
	Let $\hat{x}$ be a Darboux transformation of a finite spacelike curve $x$ and suppose that $\hat{x}$ blows up as $t \rightarrow t_*$.
	If the common circle congruence is a nonlinear circle at $t=t_*$, then $\mathcal{P}(\hat{x}(t))$ blows up at the null infinity. Otherwise, the common circle congruence is a line and $\mathcal{P}(\hat{x})$ blows up at the spatial infinity.
\end{theorem}
\begin{proof}
	For any Darboux pair, we have the common circle congruence. From Lemma~\ref{Lemma:NullInfinity}, as long as $c_1, c_2$ and $r$ are finite with uniform bounds, the entirety of the common circle congruence lives in the Penrose diagram and is bounded away from the spatial and timelike infinities. Therefore, Lemma~\ref{Lemma:NullInfinity} implies that $\mathcal{P}(\hat{x})$ blows up at the null infinity.
	When $r$ is not bounded at $t_*$, the common circle congruence becomes a line which meets the spatial infinity. Because $\mathcal{P}(\hat{x})$ lies in nearby circles for $t$ close to $t_*$, it follows that $\mathcal{P}(\hat{x})$ also blows up at the spatial infinity.
	Thus, the theorem has been proven.
\end{proof}

\begin{theorem}\label{Theorem:NullInfinityOrthogonally}
	Let $x, \hat{x}$ be a Darboux pair such that $x$ and $\hat{x}$ are simultaneously arc-length polarized. If $\hat{x}$ blows up as $t \rightarrow t_*$, then $\mathcal{P}(\hat{x})$ blows up at the null infinity in the non-orthogonal lightlike direction as $t \rightarrow t_*$.
\end{theorem}

\begin{remark}
	Note that the non-orthogonal lightlike direction will appear to meet the boundary of the Penrose diagram at a $90^{\circ}$ angle with respect to the Euclidean metric.
\end{remark}

\begin{proof}
	We suppose we have a simultaneously arc-length polarized Darboux pair, and so, without loss of generality, for the spectral parameter $\lambda \in \mathbb{R}^\times$, we have $|\hat{x}-x|^2 \equiv 1/\lambda$ and $|x'|^2=|\hat{x}'|^2$.

	From Lemma~\ref{Lemma:Radius}, we have
	\[\xi = \frac{|\hat{x}-x|^2}{2\langle \hat{x}-x,n \rangle}=\frac{1}{2\lambda\langle \hat{x}-x,n \rangle}.\]
	First, we assume $\lambda$ is positive.
	Since Lorentz-transformations preserve infinity points and length, we can rotate $n(t)$ so that $n(t_*)=(0, \alpha)$ for some finite $\alpha \in \mathbb{R}^{\times}$.

	Writing $\hat{x}=(\hat{x}_1, \hat{x}_2)$, we have
	\[\lim_{t \rightarrow t_*}\langle \hat{x}-x,n \rangle =-\alpha(\hat{x}_2-x_2).\]
	If $\hat{x}_2=x_2$, then $\hat{x}_2$ is finite and $|\hat{x}-x|^2=(\hat{x}_1-x_1)^2=1/\lambda<\infty$. However, this contradicts that $\hat{x}$ blows up, so we also have $\langle \hat{x}-x,n \rangle \neq 0$ at $t_*$. Therefore the common circle congruence is not a line at $t_*$.

	In the case of $\lambda$ is negative, since both $\hat{x}-x$ and $n$ are timelike, we have $\langle \hat{x}-x,n \rangle \neq 0$ at $t_*$ as well. Therefore by Lemma~\ref{Lemma:NullInfinity}, we find that $\hat{x}$ blows up at the null infinity.

	In particular, since $\hat{x}$ blows up at the null infinity as $t \rightarrow t_*$, both of $\hat{x}_1$ and $\hat{x}_2$ blow up, that is, we have
	\[\lim_{t \rightarrow t_*}\xi = \lim_{t \rightarrow t_*} -\frac{1}{2\lambda\langle \hat{x}-x,n \rangle} = 0.\]

	Next, without loss of generality, we suppose that $\mathcal{P}(\hat{x})$ blows up at the upper right null infinity as $t \rightarrow t_*$ so that $\hat{x}=(\hat{x}_1, \hat{x}_2)$ satisfies
	\[\lim_{t \rightarrow t_*}(\hat{x}_1+\hat{x}_2)=\infty \quad \text{and} \quad \lim_{t \rightarrow t_*}(\hat{x}_1-\hat{x}_2)\neq \pm \infty.\]

	With the $\mathbb{R}^{1,1}$ inversion $\sigma=\frac{x}{|x|^2}$, we have
	\begin{align*}
		|\mathcal{P}(\sigma(\hat{x}))'|^2 & = \frac{4|\hat{x}'|^2}{(1+(\hat{x}_1+\hat{x}_2)^2)(1+(\hat{x}_1-\hat{x}_2)^2)}\\
		& = \frac{4|x'|^2}{(1+(\hat{x}_1+\hat{x}_2)^2)(1+(\hat{x}_1-\hat{x}_2)^2)}\\
		& \xrightarrow{t \rightarrow t_*}0.
	\end{align*}

	Since both inversions and $\mathcal{P}$ are conformal, if $\hat{x}$ blows up as $t \rightarrow t_*$, then $\mathcal{P}(\hat{x})$ blows up at the null infinity in a lightlike direction as $t \rightarrow t_*$.

	The proof of non-orthogonal lightlike direction of $\mathcal{P}(\hat{x})$ remains.
	Now, we have
	\[\xi =-\frac{1}{2\lambda x_1'(\hat{x}_2-x_2)} \xrightarrow{t \rightarrow t_*}0.\]
	In the proof of Lemma~\ref{Lemma:Radius}, we had the center of a common circle congruence as
	\[x-\frac{n}{\sqrt{-|n|^2}}\xi.\]
	If $\hat{x}$ blows up at the upper right null infinity tangentially, then the center of common circle congruence also gets close to the infinity as $t \rightarrow t_*$.
	However, this contradicts that
	\[\lim_{t \rightarrow t_*}\left(x-\frac{n}{\sqrt{-|n|^2}}\xi\right)=x\]
	is finite. Thus we have proved the claim.
\end{proof}

\begin{figure}[hbt]
	\centering
	\begin{minipage}[b]{0.35\columnwidth}
		\centering
		\includegraphics[keepaspectratio, width=\columnwidth]{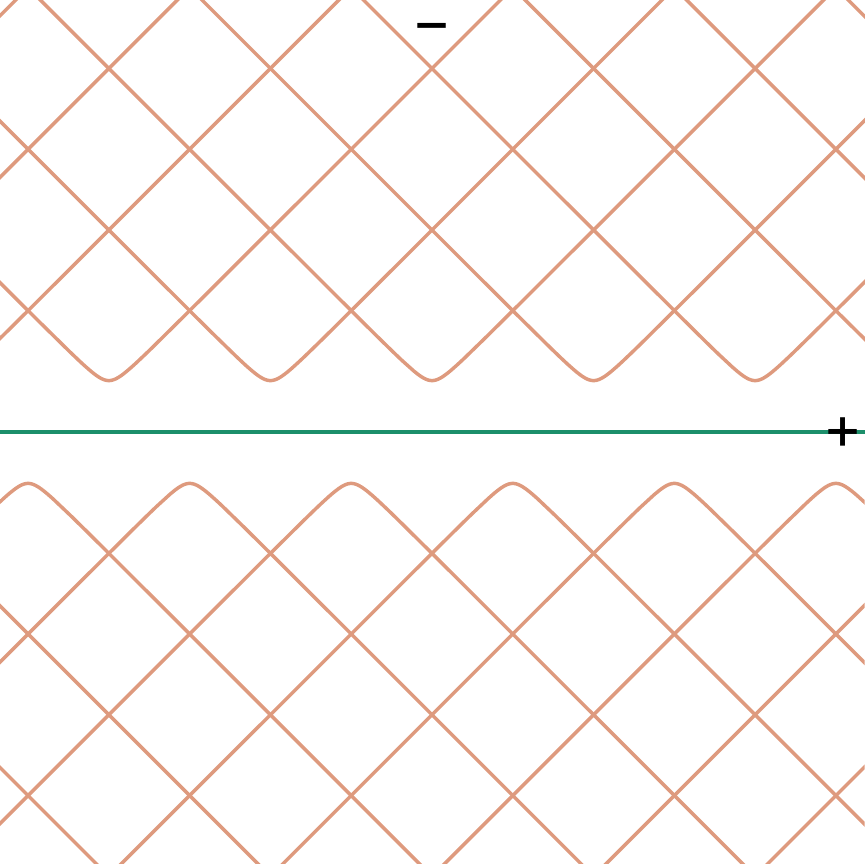}
		\subcaption{\upshape In $\mathbb{R}^{1,1}$.}\label{ALP}
	\end{minipage}
	\hspace{0.1\columnwidth}
	\begin{minipage}[b]{0.35\columnwidth}
		\centering
		\includegraphics[keepaspectratio, width=\columnwidth]{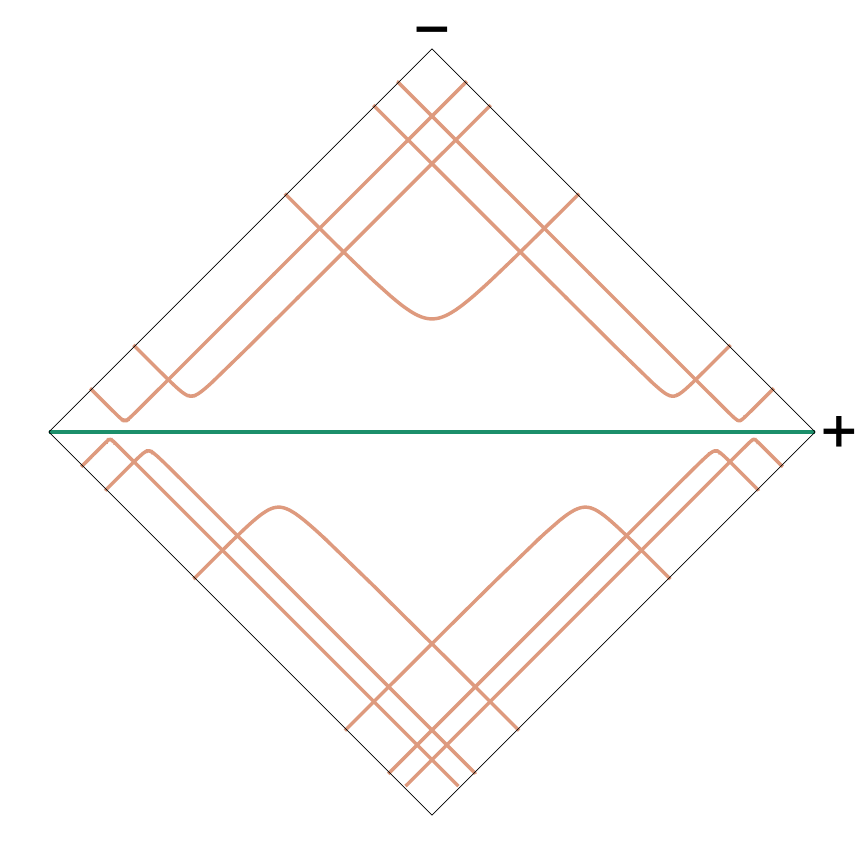}
		\subcaption{\upshape In the Penrose diagram.}\label{PenroseALP}
	\end{minipage}
	\caption{A simultaneously arc-length polarized Darboux pair.}\label{Picture:ArcLengthPolarization}
\end{figure}

\subsection{Comments on type-changing curves (future work)}
In this section, we give some comments on general curves, that is, type-changing curves.
We can apply facts about Darboux transformations of spacelike (or timelike) curves in non-lightlike parts of general curves, however, we need to pay special attention to the lightlike points.

Take a general curve $x(t)$ with a lightlike point $t=t_*$ defined on $\left[a,b\right]\subset \mathbb{R}$. If we choose an initial point at $t=a$, then we can solve the differential equation~\eqref{Equation:CrossRatio} and its Lorentz-Darboux transformation on $\left[a,t_*-\epsilon\right]$ is determined for any $\epsilon>0$.

After that, we again solve the differential equation on $\left[t_* + \epsilon,b\right]$ with an initial condition $\hat{x}(t_* + \epsilon)= \hat{x}(t_*-\epsilon)$. Finally, we take limits as $\epsilon \rightarrow 0$.
We find solutions at $t_*$ by putting together such limits, and adopt this as a definition of a Darboux transformation of general curves.

\begin{remark}
	Since the arc-length polarization becomes infinity at lightlike points, for a Darboux transformation of general curves, we have no choice but to allow the polarization $m(t)$ to take values in $\mathbb{R}^{\times} \cup \{\infty\}$.
	Therefore, we also generalize the polarization to define Darboux transformations of general curves. See also \cites{MR3936232}.
\end{remark}

Let us explore the behavior of Darboux transformations at lightlike points. Let $x, \hat{x}$ be a smooth Darboux pair with an arc-length polarization $m(t):I \rightarrow \mathbb{R}^{\times} \cup \{\infty\}$ of $x$, and let $t_* \in I$ be a point such that $x'$ becomes lightlike at $t=t_*$. Then we have $\lim_{t \rightarrow t_*}|x'(t)|^2=0$ and we have
\[\frac{x'\hat{x}'}{|x'|^2}=\lambda (\hat{x}-x)^2.\]
According to the values of $\hat{x}'$ and $(\hat{x}-x)^2$, we need to use L'H$\hat{\mathrm{o}}$pital's rule on the split complex plane.
Since $x$ and $\hat{x}$ are smooth, L'H$\hat{\mathrm{o}}$pital's rule does apply. However, note that unlike 2-dimensional fields, in the split complex plane we use it when the limits get close to not only $0$, but also zero divisors.

\begin{figure}[hbt]
	\centering
	\begin{minipage}[b]{0.375\columnwidth}
		\centering
		\includegraphics[keepaspectratio, width=\columnwidth]{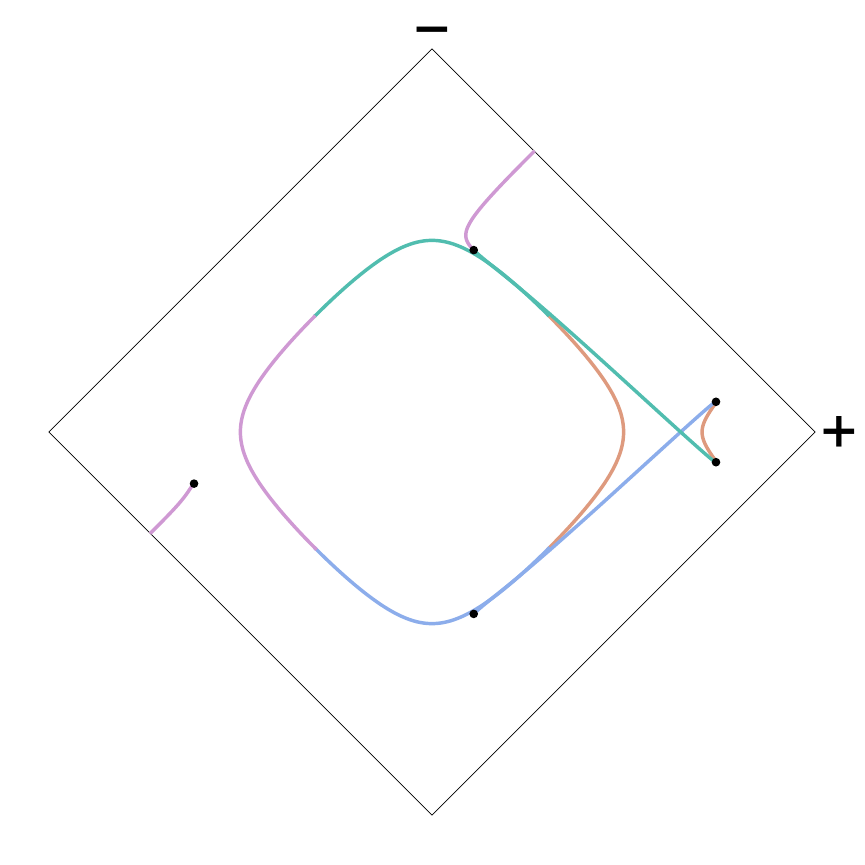}
		\subcaption{\upshape $\lambda >0$ case.}\label{PenroseGeneral1}
	\end{minipage}
	\hspace{0.05\columnwidth}
	\begin{minipage}[b]{0.375\columnwidth}
		\centering
		\includegraphics[keepaspectratio, width=\columnwidth]{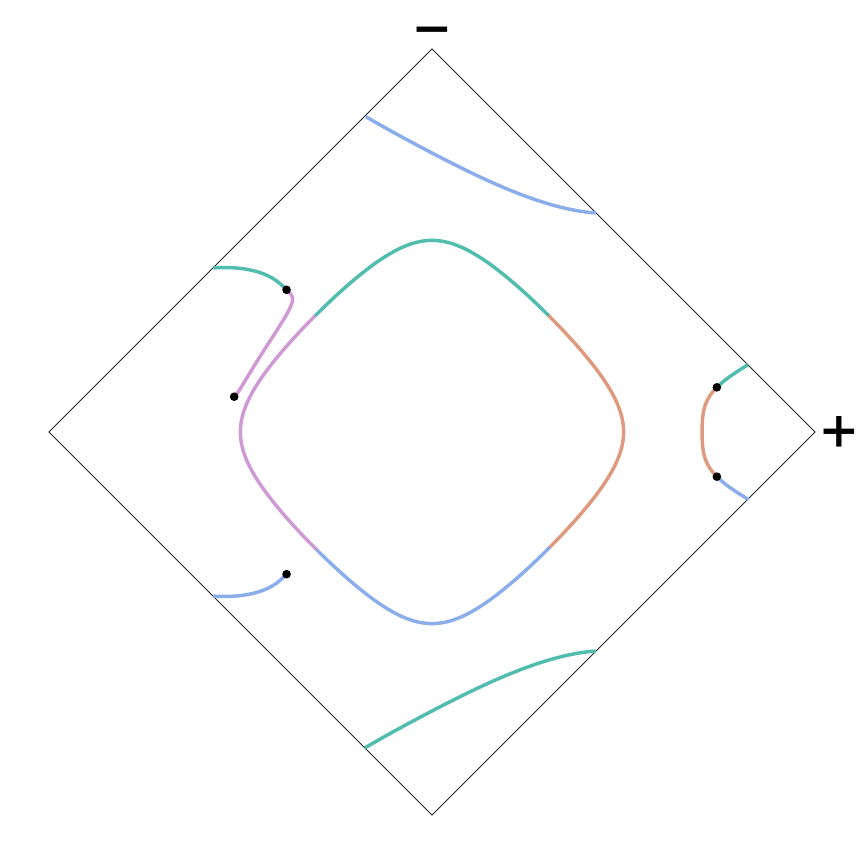}
		\subcaption{\upshape $\lambda <0$ case.}\label{PenroseGeneral2}
	\end{minipage}
	\caption{Darboux transformations of an embedded Euclidean circle.}\label{Picture:DarbouxTransformationOfCircle}
\end{figure}

As examples, in Figure~\ref{Picture:DarbouxTransformationOfCircle}, the round closed curve in the center is an embedded Euclidean circle. Each of the colored parts correspond and black points on the Darboux transformations correspond to the lightlike points of the embedded circle. These figures are given numerically, and we seem to observe singular points at lightlike points.

\vspace{10pt}

\noindent\textbf{Acknowledgements.}
I would like to express my gratitude to my adviser Wayne Rossman for his support and continued feedback throughout this work, and my appreciation for valuable conversations with Udo Hertrich-Jeromin, Katrin Leschke, Masashi Yasumoto, Joseph Cho and Denis Polly.
Further, this work was supported by JST SPRING, Grant Number JPMJSP2148.

\nocite{*}
\bibliography{./references.bib}

\end{document}